\documentclass[11pt,a4paper]{article}

\usepackage{amsmath, amsfonts, amssymb}

\def\be#1\ee{\begin{equation}#1\end{equation}}

\newtheorem{thm}{Theorem}[section]

\newtheorem{prop}[thm]{Proposition}

\newtheorem{example}[thm]{Example}
\newtheorem{rem}[thm]{Remark}

\DeclareMathOperator{\Var}{Var}

\def\R{\mathbb{R}}
\def\E{\mathbb{E}\,}

\def\Z{{\mathbb Z}}
\def\N{{\mathbb N}}

\newenvironment{proof}[1][] {\noindent {\bf Proof#1:} }{\hspace*{\fill}$\square$\medskip\par}

\def\span#1{\overline{{\rm span}}\{#1\}}    

\def\al{{\alpha}}
\def\AA{{\mathcal A}}
\def\b{{\beta}}
\def\C{{\mathbb C}}

\def\cov{\textrm{Cov}}
\def\D{{\mathbb D}}
\def\DD{{\mathcal D}}

\newcommand{\eps}{\varepsilon}
\def\HH{{\mathcal H}}

\def\LL{{\mathcal L}}

\def\SS{{\mathcal S}}

\def\W{{\mathcal W}}

\def\B{{B}}
\def\K{{K}}
\def\V{{V}}
\def\X{{X}}
\def\tX{\widetilde{\X}}

\def\HTO{H^\circ}
\def\LLO{\LL^\circ}

\def \=L{\ {\buildrel\hbox{\scriptsize d }\over =}\ }

\begin{document}

\title{\bf Some extensions of linear approximation and prediction
problems for stationary processes}
\author{
    Ildar Ibragimov\footnote{
    St.Petersburg Department of V.A.Steklov Institute of Mathematics of Russian
    Academy of Science, Fontanka embankment, 27, St.Petersburg, 191023, Russia,
    and St.Petersburg State University, 7/9 Universitetskaya nab., St.Petersburg,
   199034 Russia email {\tt ibr32@pdmi.ras.ru}}
    \and
    Zakhar Kabluchko\footnote{
    M\"unster University, Orl\'eans-Ring 10, 48149 M\"unster, Germany,
    email \ {\tt zakhar.kabluchko@uni-muenster.de}}
    \and
   Mikhail Lifshits\footnote{
   St.Petersburg State University, 7/9 Universitetskaya nab., St.Petersburg,
   199034 Russia, email {\tt mikhail@lifshits.org} and MAI, Link\"oping
   University.}
}
\date{\today}

\maketitle

\begin{abstract}
Let $(\B(t))_{t\in \Theta}$ with $\Theta=\Z$ or $\Theta=\R$ be a wide sense
stationary process with  discrete or continuous time. The classical linear prediction
problem consists of finding an element in $\span{B(s),s\le t}$ providing the best
possible mean square approximation to the variable $\B(\tau)$ with $\tau>t$.

In this article we investigate this and some other similar problems where, in
addition to prediction quality, optimization takes into account other features
of the objects we search for. One of the most motivating examples of this kind
is an approximation of a stationary process $\B$ by a stationary differentiable process
$\X$ taking into account the kinetic energy that $\X$ spends in its
approximation efforts.
\end{abstract}
\vskip 1cm

\bigskip

\vfill

\newpage

\section{Introduction and problem setting}

\subsection{Motivating example}
Let $(\B(t))_{t\in \Theta}$ with $\Theta=\Z$ or $\Theta=\R$
be a wide sense stationary process with discrete or continuous
time. The classical linear prediction problem consists of
finding  an element in $\span{B(s),s\le t}$
providing the best possible mean square approximation to
the variable $\B(\tau)$ with $\tau>t$, see \cite{Kolm} and
\cite{Doob1, Doob2, GS, Roz, Yagl, W}.

Below we investigate this and some other similar problems where, in
addition to prediction quality, optimization takes into account
other features of the objects we search for, such as the smoothness
properties of approximation processes.

Here and elsewhere in the article $\span{\cdot}$ stands for the closed
linear span of a set in a Hilbert space.
All mentioned processes are assumed to be complex valued and all Hilbert spaces
are assumed to be complex.

\begin{example} \label{ex:kinetic} {\rm (Approximation saving kinetic
energy, \cite{KabLi}). By the instant {\it kinetic energy} of a process
$(\X(t))_{t\in \R}$ we understand  just its squared derivative
$|\X'(t)|^2$. It is more than natural to search for an approximation
of a given stationary process $(\B(t))_{t\in \R}$ by a differentiable
stationary process $(\X(t))_{t\in \R}$ taking
into account the kinetic energy that $\X$ spends in its approximation
efforts. The goals of the approximation quality and energy saving may
be naturally combined with averaging in time by minimization of the
functional
\[
   \lim_{N\to\infty} \frac 1N\ \int_0^N \left[ |\X(t)-\B(t)|^2 +
   \al^2|\X'(t)|^2\right] dt.
\]
Here $\al>0$ is a fixed scaling regularization parameter balancing the quality
of approximation and the spent energy.

If, additionally, the process $\X(t)-\B(t)$ and the derivative $\X'(t)$
are stationary processes in the strict sense, in many situations
ergodic theorem applies and the limit above is equal to
$\E |\X(0)-\B(0)|^2 + \al^2 \E |\X'(0)|^2$.

Therefore, we  may simplify our task to solving the problem
\be \label{EE0}
   \E |\X(0)-\B(0)|^2 + \al^2 \E |\X'(0)|^2 \to \min,
\ee
and setting aside ergodicity issues.

The problem \eqref{EE0} makes sense either in a simpler
{\it linear non-adaptive setting}, i.e. with
\[
   \X(t) \in \span{\B(s), s\in \R},  \quad t\in \R,
\]
or in {\it linear adaptive setting} by requiring additionally
\[
   \X(t) \in \span{\B(s), s\le t},  \quad t\in \R.
\]
In other words, this means that we only allow approximations based on
the current and past values of $\B$.
}
\end{example}
\bigskip

Let us start with a basic notation. Let $(\B(t))_{t\in \Theta}$ be
a complex-valued random process satisfying
$\E\B(t)=0$, $\E|\B(t)|^2<\infty$ for all $t\in \Theta$.

Consider $H:=\span{\B(t), t\in \Theta}$ as a Hilbert space equipped
with the scalar product $(\xi,\eta)=\E (\xi\overline{\eta})$. For
$T\subset \Theta$ let $H(T):= \span{\B(t), t\in T}$.

Furthermore, let $L$ be a linear operator with values in $H$ and
defined on a linear subspace $\DD(L)\subset H$. For a fixed
$\tau\in \Theta$, consider the extremal problem
\be\label{A}
   \E|Y-\B(\tau)|^2 + \E|L(Y)|^2 \to \min,
\ee
where the minimum is taken over all $Y\in H(T) \bigcap \DD(L)$.
The first term in the sum describes approximation, prediction, or
interpolation quality while the second term stands for additional
properties of the object we are searching for, e.g. for the
smoothness of the approximating process.

This is the most general form of the problem we are interested in.
Below we specify the class of the considered processes to
one-parametric wide sense stationary processes with discrete or continuous
time, introduce appropriate class of operators $L$ and
explain in Subsection \ref{ss:kin_as_L} why Example \ref{ex:kinetic}
is a special case of problem \eqref{A}.

\subsection{Spectral representation: brief reminder}

Let now $(\B(t))_{t\in \Theta}$ be the main object of our investigation
-- a centered wide sense stationary random process with univariate discrete
($\Theta=\Z$) or continuous ($\Theta=\R$) time. In case of continuous
time we additionally assume that $\B$ is mean square continuous.

Here and elsewhere we assume that all random variables under consideration
are centered. In particular, $\E \B(t)=0$ for all $t\in \Theta$.

By Khinchin theorem, the covariance function
\[
   \K(t) := \E \B(t)\overline{\B(0)}
\]
admits the spectral representation
\[
   \K(t) := \int e^{itu} \mu(du).
\]
Here and in the sequel integration in similar integrals is performed
over the interval $[-\pi,\pi)$ in case of the discrete time
processes and over real line $\R$ in case of the continuous time processes.
The finite measure $\mu$ on $[-\pi,\pi)$ or on $\R$,
respectively, is the spectral measure of the process $\B$. The process
$\B$ itself admits the spectral representation as a stochastic integral
\be \label{specBs}
   \B(t) = \int e^{itu} \W(du)
\ee
where $\W$ is an orthogonal random measure on $[-\pi,\pi)$ or on $\R$,
respectively, with $\E|\W(A)|^2=\mu(A)$.

Let
\[
   \LL = L_2(\mu)
       =\left\{\phi:\ \int |\phi(u)|^2 \mu(du) <\infty \right\}
\]
be equipped with the usual scalar product
\[
   \left(\phi, \psi \right)_\LL
   = \int \phi(u)\, \overline{\psi(u)}\, \mu(du).
\]
Recall that for any
\[
  \xi=  \int \phi(u)\, \W(du), \qquad \eta=  \int \psi(u)\, \W(du)
\]
it is true that
\[
  \E \xi\overline{\eta} = \int \phi(u)\, \overline{\psi(u)}\, \mu(du)
  =   \left(\phi, \psi \right)_\LL.
\]
It follows that the correspondence $\B(t) \rightleftarrows e^{itu}$
extends to the linear isometry between $H$ and the closed linear
span of the exponents in $\LL$. Actually, the latter span coincides
with entire space $\LL$, cf.\,\cite[Section 1.9]{GS}, and we obtain a linear
isometry between $H$ and $\LL$ provided by stochastic integral.
In other words, every element $\xi$ of Hilbert space $H$ can be
represented as a stochastic integral
\be \label{specxi}
   \xi = \int \phi_\xi(u) \W(du)
\ee
with some complex valued function $\phi_\xi\in\LL$, and every
random variable $\xi$ admitting the representation \eqref{specxi}
belongs to $H$.
\bigskip

An analogous theory exists for processes with wide sense
stationary increments.
Let $\B$ be such process with zero mean (for continuous time we
additionally assume mean square continuity). Similarly to
\eqref{specBs}, the process $\B$ admits a spectral representation
\[
    \B(t) = \B_0 + \B_1 t + \int (e^{itu}-1) \W(du)
\]
where $\W(du)$ is an orthogonal random measure controlled by the spectral
measure $\mu$ and $\B_0,\B_1$ are centered random variables uncorrelated
with $\W$, see \cite[p.213]{Yagl}. Notice that in case of processes with
wide sense stationary increments the spectral measure $\mu$ need not be finite but
it must satisfy $\mu\{0\}=0$ and L\'evy's integrability condition
\[
   \int_{-\pi}^\pi u^2\, \mu(du) < \infty
\]
for discrete time and
\[
   \int_\R \min\{u^2,1\} \mu(du) < \infty
\]
for continuous time.

In the following we let $B_0=B_1=0$ because for prediction problems
we handle here the finite rank part is uninteresting. We also do not
loose any interesting example with this restriction.  Therefore, we
consider the processes
\[
    \B(t) = \int (e^{itu}-1) \W(du).
\]

\subsection{Probabilistic problem setting}

The operators $L$ we are going to handle are those of the form
\be \label{L}
   L\xi = L\left( \int \phi_\xi(u) \W(du) \right)
   := \int \ell(u) \phi_\xi(u) \W(du),
\ee
where $\ell$ is a measurable function on $\R$ or on $[-\pi,\pi)$,
respectively. The domain $\DD(L)$ consists of $\xi$ such that
\[
  \E|L\xi|^2 = ||L\xi||_H^2
  = \int |\ell(u) \phi_\xi(u) |^2 \mu(du)<\infty.
\]
Such operators are often called {\it linear filters} while the
function $\ell$ is called the {\it frequency characteristic} of
a filter.

Below we consider problem \eqref{A} applied to wide sense stationary processes
with discrete or continuous time and operators $L$ from \eqref{L}.
For the space $H(T)$ we consider a variety of choices. Most
typically, we take $H(T)=H:=\span{\B(s), -\infty <s<\infty}$, the
space generated by all variables, or
$H(T)=H_t:=H((-\infty,t])= \span{\B(s),s\le t}$, the space
generated by the past of the process, or
$H(T)=\HTO_t:=\span{\B(s),|s|\ge t}$.

In problem \eqref{A}, we take the value of $\B$ at some point $\tau$ as
a subject of approximation. When $\B$ is a wide sense stationary process,
we may take $\tau=0$ without loss of generality.

Therefore, three following variations of problem \eqref{A} are
considered below.

Problem I (approximation):
\[
   \E |Y-\B(0)|^2+ \E|LY|^2 \to \min, \qquad Y\in H.
\]

Problem II (prediction):
\[
    \E |Y-\B(0)|^2+ \E|LY|^2 \to \min, \qquad Y\in H_t.
\]

Problem III (interpolation):
\[
    \E |Y-\B(0)|^2+ \E|LY|^2 \to \min, \qquad Y\in \HTO_t.
\]

Notice that, due to the presence of $L$, Problems II and III represent
an extension of the classical prediction and interpolation problems.
As for Problem I, once $L$ is omitted, it is trivial. In our setting
it is also easy but provides non-trivial results (even in the simplest
cases) and therefore is sufficiently interesting.

Sometimes we call the setting of Problem II adaptive, because the best
approximation is based on (adapted to) the  known past values of the
process. Opposite to this, the setting of Problem I is called
non-adaptive.

In the classical case, i.e. with $L=0$, Problem II, as stated here, is
non-trivial only for negative $t$. However, in presence of $L$ it
makes sense for arbitrary $t$.
\medskip

\subsection{Analytic problem setting}

Due to the spectral representation \eqref{specxi} problems I -- III
admit the following analytic setting.

Problem I$'$:
\be \label{P1}
   \int |\psi(u)-1|^2 \mu(du) + \int |\ell(u)\psi(u)|^2 \mu(du)
   \to \min, \qquad \psi\in \LL.
\ee

Problem II$'$:
\be \label{P2}
   \int |\psi(u)-1|^2 \mu(du) + \int |\ell(u)\psi(u)|^2 \mu(du) \to \min,
   \  \psi\in\span{e^{isu}, s\le t}.
\ee

Problem III$'$:
\be \label{P3}
   \int\! |\psi(u)-1|^2 \mu(du)\!
   +\! \int\! |\ell(u)\psi(u)|^2 \mu(du) \to \min,
   \  \psi\in\span{e^{isu}, |s|\ge t}.
\ee
The spans in Problems II$'$ and III$'$ are taken in $\LL$.

\subsection{Energy saving approximation as a special case of
extended prediction problem}
\label{ss:kin_as_L}

Consider the setting of Example \ref{ex:kinetic}:
given a zero mean wide sense stationary process $B=(\B(t))_{t\in\R}$
with spectral representation \eqref{specBs}, the problem is to
minimize the functional
\[
  \E |\X(0)-\B(0)|^2 + \al^2 \E |\X'(0)|^2
\]
over all mean square differentiable processes $X=(\X(t))_{t\in\R}$ such
that the processes $\X$ and $\B$ are jointly wide sense stationary.
The latter means that each of them is wide sense stationary and also the
cross-covariance $\E\X(t)\overline{\B(s)}$ depends only on $t-s$.

First of all, we show that while solving this minimization problem
one may only consider approximating processes of special type,
namely,
\be \label{psistat}
  \tX(t)=\int e^{itu} \psi(u)\W(du), \qquad \psi\in \LL.
\ee
Indeed, for arbitrary $\X$, we may decompose its initial value as
$\X(0)=\X^\bot +\tX(0)$ with $\tX(0)\in H$, $\X^\bot$ orthogonal to
$H$. By representation \eqref{specxi} there exists $\psi\in\LL$ such
that
\[
   \tX(0)= \int  \psi(u)\, \W(du).
\]
For this $\psi$ define the process $\tX(t)$ by \eqref{psistat}.
We show that the process $\tX$ is at least as good as $\X$ in the
sense of \eqref{EE0}.

Due to the joint wide sense stationarity, for any $s,t$ we have
\begin{eqnarray*}
  \E \X(t)\overline{\B(s)}&=& \E \X(0)\overline{\B(s-t)}
  =\E \tX(0)\overline{\B(s-t)}
  =\int \psi(u)e^{i(t-s)u} \mu(du);
\\
  \E \tX(t)\overline{\B(s)}&=& \int [e^{itu}\psi(u)]e^{-isu}\mu(du)
  = \int \psi(u)e^{i(t-s)u} \mu(du).
\end{eqnarray*}
It follows that $\X(t)-\tX(t)$ is orthogonal to $\B(s)$ for each $s$,
hence, it is orthogonal to $H$. Furthermore, it is easy to show
that if $\X$ is mean square differentiable then so are its components
$\tX$ and $\X-\tX$. For their derivatives, we know that
\[
   \tX'(t)= \int e^{itu}(iu)\,\psi(u)\, \W(du) \in H
\]
and $(\X-\tX)'(t)$ is orthogonal to $H$. Hence,
\begin{eqnarray*}
   &&\E |\X(0)-\B(0)|^2 + \al^2 \E |\X'(0)|^2
\\
   &=& \E |(\X(0)-\tX(0))+ (\tX(0)-\B(0))|^2
       +  \al^2 \E |(\X'(0)-\tX'(0))+ \tX'(0)|^2
\\
   &=& \E |\X(0)-\tX(0)|^2 +\E |\tX(0)-\B(0)|^2
\\
   &&    +  \al^2 \E |(\X'-\tX')(0)|^2   +  \al^2 \E |(\tX'(0)|^2
 \\
   &\ge&  \E  |\tX(0)-\B(0)|^2 +  \al^2 \E |\tX'(0)|^2.
\end{eqnarray*}
Therefore, $\tX$ is at least as good for \eqref{EE0}, as $\X$.

Finally, notice that for the processes defined by \eqref{psistat}
the expression in \eqref{EE0} is equal to
\[
   \int |\psi(u)-1|^2 \mu(du) + \int |\ell(u)\psi(u)|^2 \mu(du)
\]
with $\ell(u)=\al iu$, exactly as in the analytical versions of our
problems \eqref{P1} and \eqref{P2}. In the non-adaptive version
of the approximation problem we have to optimize over $\LL$, as
in \eqref{P1}, while for adaptive version the requirement
$\tX\in H_t$ for all $t$ is satisfied iff
$\psi\in\span{e^{isu}, s\le 0}$, as in \eqref{P2}.

One may consider other types of energy, e.g. based on higher
order derivatives of $\X$. This option leads to the same problems
with arbitrary polynomials $\ell$.

For discrete time case it is natural to replace the derivative
$\X'$ by the difference $\X(1)-\X(0)$. Then we obtain the same
problem with $\ell(u)=\al(e^{iu}-1)$, for $u\in[-\pi,\pi)$.

Examples of optimal non-adaptive energy saving approximation are given
in Section \ref{s:ex} below.

One may also consider the energy saving approximation for the
processes with wide sense stationary increments. Consider such a process
$\B$ and its approximation $\X$ such that $(\X(t),\B(t))_{t\in \Theta}$ with
$\Theta=\Z$ or $\Theta=\R$ is a
two-dimensional process with wide sense stationary increments and
$(\X(t)-\B(t))_{t\in \Theta}$ is a wide sense stationary process.
Since $\B(0)=0$, the analogue of \eqref{EE0} is
\be \label{EE0si}
  \E |\X(0)|^2+ \al^2 \E|\X'(0)|^2 \to \min.
\ee
Similarly to the case of  wide sense stationary processes one can show that,
analogously to \eqref{psistat}, it is sufficient only to consider
approximating processes of the special form
\[
  \tX(t):=\int \left( e^{itu} \psi(u)-1\right) \W(du),
  \qquad \psi-1\in \LL.
\]
Then problem \eqref{EE0si} takes familiar analytical form
\[
  \int |\psi(u)-1|^2 \mu (du) +  \int |\psi(u) (iu)|^2 \mu (du)
  \to\min
\]
with requirements $\psi-1\in \LL$ for non-adaptive setting and
\[
   e^{itu}\psi(u)-1\in \span{e^{isu}-1, s\le t}
\]
in adaptive setting. The latter may be simplified
to
\[
   \psi-1\in \span{e^{isu}-1, s\le 0}.
\]
\bigskip

\section{Abstract Hilbert space setting}

The basic matters about our problems such as the existence of
the solution or its uniqueness are easier to handle in a more
abstract setting.
A formal extension of problem \eqref{A} looks as follows. Let
$H$ be a separable Hilbert space with the corresponding scalar
product $(\cdot,\cdot)$ and norm $||\cdot||$. Let $L$ be
a linear operator taking values in $H$ and defined on a linear
subspace $\DD(L)\subset H$. Consider a problem
\be \label{A1}
  G(y):= ||y-x||^2 +||L y||^2 \to \min.
\ee
Here $x$ is a given element of $H$ and minimum is taken over all
$y\in H_0 \bigcap \DD(L)$ where $H_0\subset H$ is a given closed
linear subspace.

The following results are probably well known, yet for
completeness we give their proofs in Section \ref{s:HSproofs}.

\begin{prop}\label{p:exist}
    If $L$ is a closed operator
then the problem \eqref{A1} has a solution $\xi\in H_0 \bigcap \DD(L)$.
\end{prop}

\begin{prop}\label{p:uniq}
    The problem \eqref{A1} has at most one solution.
\end{prop}

\begin{rem} {\rm
   Unlike Proposition $\ref{p:exist}$, the assertion of
   Proposition $\ref{p:uniq}$ holds without additional assumptions
   on the operator $L$.
}
\end{rem}

\begin{prop} \label{p:LL}
   Assume that in problem \eqref{A1}  we have $H_0=H$ and $L$ is
   a closed operator with the domain dense in $H$.
   Then the unique solution of  \eqref{A1} exists and is given by the formula
\[
   \xi=(I+L^*L)^{-1} x,
\]
where $I:H\to H$ is the identity operator.
\end{prop}

\begin{prop} \label{p:euler} If $\xi$ is a solution of problem \eqref{A1},
then $\xi$ provides the unique solution of equations
\be \label{euler}
  (\xi-x,h)+(L\xi,L h)=0 \qquad  \textrm{for all } h\in H_0\cap \DD(L).
\ee
\end{prop}

\begin{rem} {\rm
If the operator $L$ is bounded, one may rewrite equations \eqref{euler}
as
\[
  ((I+L^*L)\xi -x, h)=0 \qquad \textit{for all } h\in H_0,
\]
where $I$ is the identity operator in $H$.
}
\end{rem}

\section{Solution of the non-adaptive problem}

\begin{thm} \label{t:nonad}  Let $\B$ be a centered wide sense
stationary process with discrete or continuous time. Let $L$ be a
linear filter \eqref{L} with arbitrary measurable frequency characteristic
$\ell(\cdot)$. Then the unique solution of Problem I exists and is given by the
formula
\be \label{sol_na}
  \xi=\int \frac{1}{1+|\ell(u)|^2} \ \W(du).
\ee
The error of optimal approximation, i.e. the minimum in Problem I,
and in its equivalent form \eqref{P1}, is given by
\begin{eqnarray} \label{errna}
    \sigma^2 &:=& \E |\xi-\B(0)|^2 + \E|L\xi|^2
    = \int \frac{|\ell(u)|^2}{1+ |\ell(u)|^2} \, \mu(du)
\\ \nonumber
    &=& \E |\B(0)|^2 - \int \frac{\mu(du)}{1+ |\ell(u)|^2} \ .
\end{eqnarray}
\end{thm}

\begin{proof}
The operators of type \eqref{L} are clearly closed and have a dense domain.
Therefore, Proposition \ref{p:exist} provides the existence of
solution. Furthermore, Proposition \ref{p:uniq} confirms that the
solution is unique.  Proposition \ref{p:LL} states the form of the
solution
\[
   \xi=(I+L^*L)^{-1} \B(0).
\]
By using the definition of $L$, for any $Y\in H$ we easily obtain
\[
 (I+L^*L)^{-1} \, Y
 =  \int  \frac{1}{1+|\ell(u)|^2}\ \phi_Y(u)\, \W(du).
\]
For $Y=\B(0)$ we have $\phi_Y(u)\equiv 1$, thus \eqref{sol_na}
follows. Finally, by isometric property,
\begin{eqnarray*}
\sigma^2 &=& \E |\xi-\B(0)|^2 + \E|L\xi|^2
\\
    &=& \int \left[ \left| \frac{1}{1+|\ell(u)|^2}-1\right|^2
    + \left|\frac{\ell(u)}{1+|\ell(u)|^2}\right|^2 \right] \, \mu(du)
\\
    &=&  \int \frac{|\ell(u)|^4+|\ell(u)|^2}{(1+|\ell(u)|^2)^2}  \, \mu(du)
     =  \int \frac{|\ell(u)|^2}{1+|\ell(u)|^2}  \, \mu(du),
\end{eqnarray*}
as claimed in \eqref{errna}.
\end{proof}

\begin{rem} {\rm
For equivalent Problem I$'$ one can arrive to the same conclusion
as in Theorem \ref{t:nonad} in a fairly elementary way. Using the
full square identity
\[
  |\psi(u)-1|^2 +  |\ell(u)\psi(u) |^2 =
  \left(1+ |\ell(u)|^2 \right)
  \left| \psi(u)-\frac{1}{ 1+ |\ell(u)|^2}\right|^2
    +\frac{|\ell(u)|^2}{1+ |\ell(u)|^2}\ ,
\]
one immediately observes that
\be \label{sol_na_anal}
  \psi_\xi(u):=\frac{1}{1+ |\ell(u)|^2}
\ee
solves Problem I$'$, while the error is given by \eqref{errna}.
}
\end{rem}

\begin{rem} {\rm
For processes with  wide sense stationary increments the extremal problem
is the same, hence the solution \eqref{sol_na} is the same.
It should be noticed however that the solution is correct only
if the quantity $\sigma^2$ above is finite (for finite measure
$\mu$ it is always finite but for infinite measure and some
choices of $\ell$ it may be infinite). For example if $\ell$
is a polynomial without free term, $\ell(u)=\sum_{k=1}^m c_ku^k$,
then $\sigma^2$ above is finite. This includes kinetic energy
case $\ell(u)=i \alpha u$. Otherwise, if
\[
   \int \frac{|\ell(u)|^2}{1+|\ell(u)|^2}  \, \mu(du) =\infty,
\]
the quantity in Problem I$'$ is infinite for all admissible $\psi$.
}
\end{rem}

\section{Some examples of non-adaptive approximation}
\label{s:ex}

In this section we illustrate general results by some typical
examples. In all examples we consider kinetic energy, i.e. we let
$\ell(u)=\al(e^{iu}-1)$ in the discrete time case and
$\ell(u)=\al i u$ in the continuous time.

For discrete time we get
\begin{eqnarray*}
   |\ell(u)|^2+1 &=& \al^2(e^{iu}-1)(e^{-iu}-1) +1
   \\
   &=& \frac{\al^2}{\b} (e^{iu}-\b)(e^{-iu}-\b)
\end{eqnarray*}
where
\be \label{beta}
  \b=\frac{2\al^2+1+\sqrt{1+4\al^2}}{2\al^2}>1
\ee
is the larger root of the equation
\[
    \b^2-\frac{2\al^2+1}{\al^2}\, \b +1=0.
\]
For the integrand in the solution \eqref{sol_na} of the non-adaptive problem, we easily derive
an expansion
\begin{eqnarray} \nonumber
    \frac{1}{|\ell(u)|^2+1}
    &=& \frac{\b}{\al^2} \, \frac{1}{(e^{i u}-\b)(e^{-i u}-\b)}
   \\ \label{series1}
   &=& \frac{1}{\sqrt{1+4\al^2}} \left( 1+ \sum_{k=1}^\infty \b^{-k}
   \left( e^{i k u}+ e^{-i k u}\right)\right).
\end{eqnarray}

By plugging this expression into \eqref{sol_na}, it follows that the solution
of discrete non-adaptive problem involving kinetic energy
is given by the moving average with bilateral geometric progression weight:
\be \label{XgB_star_discr}
  \xi =  \frac{1}{\sqrt{1+4\al^2}} \left( \B(0) + \sum_{k=1}^\infty \b^{-k}
   \left( \B(k)+ \B(-k) \right)\right).
\ee
By \eqref{errna}, the error of optimal non-adaptive approximation
in the discrete time case is given by
\be \label{errna_kin_discr}
   \sigma^2  = \int_{-\pi}^{\pi}
   \frac{\al^2 |e^{iu}-1|^2}{ \al^2 |e^{iu}-1|^2 +1}\ \mu(du).
\ee

For continuous time we get similar results.
By using inverse Fourier transform, we have
\[
  \frac{1}{ |\ell(u)|^2 +1} = \frac{1}{ \al^2 u^2 +1}
  = \frac 1{2\al}\ \int_\R  \exp\{i\tau u - |\tau|/\al \} \, d\tau.
\]
By plugging this expression into \eqref{sol_na}, it follows that
the solution of continuous non-adaptive problem involving kinetic
energy is given by the moving average
\be \label{XgB_star}
   \xi = \frac 1{2\al} \int_\R \exp\{-|\tau|/\al\} \B(\tau) d\tau.
\ee
By \eqref{errna}, the error of optimal non-adaptive approximation
in the continuous time case is given by
\be \label{errna_kin}
   \sigma^2 = \int_\R  \frac{\al^2 u^2}{ \al^2 u^2 +1}\ \mu(du).
\ee

Notice that both solutions \eqref{XgB_star_discr} and
\eqref{XgB_star} are indeed non-adaptive at all because they involve
future values of $\B$. Let us also stress  that these solutions formulae
are the same for any spectral (covariance) structure of $\B$.
The formulae \eqref{XgB_star} and \eqref{errna_kin} were obtained
earlier in \cite{KabLi}.

We start with discrete time examples.


\begin{example} \label{ex:iid} {\rm
A sequence $(B(t))_{t\in \Z}$ of centered non-correlated random variables
with constant variance $V\ge 0$ has the spectral measure
\[
   \mu(du):= \frac{\V du}{2\pi}.
\]
Surprisingly, the answer to the non-adaptive problem taking kinetic
energy into account even for this sequence is already non-trivial, since
the best non-adaptive approximation is given by the series
\eqref{XgB_star_discr}. The error formula \eqref{errna_kin_discr} yields
\[
   \sigma^2  = \V \int_{-\pi}^{\pi}  \frac{\al^2 |e^{iu}-1|^2 du}{2\pi
   \al^2 |e^{iu}-1|^2 +1}
   =  \V  \left (1 - \frac{1}{\sqrt{1+4\al^2}} \right),
\]
cf. an extension below in \eqref{errna_ar}.
}
\end{example}
\smallskip


\begin{example} \label{ex:ar}
{\rm
A sequence of random variables $(\B(t))_{t\in\Z}$ is called
autoregressive, if it satisfies the equation
$\B(t)=\rho \B(t-1) + \xi(t)$, where $|\rho|<1$ and
$(\xi(t))_{t\in\Z}$ is a sequence of centered non-correlated
random variables with some variance $\V$. In this case we have
a representation
\[
   \B(t)= \sum_{j=0}^\infty \rho^j \xi(t-j), \qquad t\in \Z.
\]
Given the spectral representation
$\xi(t) = \int_{-\pi}^\pi e^{i t u} \W(du)$ from the previous
example, we obtain
\[
   \B(t)= \int_{-\pi}^\pi \sum_{j=0}^\infty \rho^j e^{i(t-j)u} \W(du)
   = \int_{-\pi}^\pi  \frac{1}{ 1- \rho\, e^{-i u}} \ e^{i t u} \W(du).
\]
We see that the spectral measure  for $\B$ is
\be \label{mu_ar}
    \mu(du):= \frac {\V du}{2\pi|1- \rho\, e^{-i u}|^2}\, .
\ee

The best non-adaptive approximation is given by the series
\eqref{XgB_star_discr}. By \eqref{errna_kin_discr} and \eqref{mu_ar},
the error of non-adaptive approximation is
\begin{eqnarray*}
  \sigma^2 &=&   \frac{\V}{2\pi}  \int_{-\pi}^\pi
  \frac{\al^2 |e^{iu}-1|^2}{ \al^2 |e^{iu}-1|^2 +1}\
  \frac {du}{|1- \rho\, e^{-iu}|^2}
  \\
  &=&   \frac{\V}{2\pi} \left[ \int_{-\pi}^\pi
        \frac {du}{|1- \rho\, e^{-iu}|^2}
        - \int_{-\pi}^\pi  \frac{1}{ \al^2 |e^{iu}-1|^2 +1} \
        \frac {du}{|1-\rho\, e^{-iu}|^2} \right].
\end{eqnarray*}
By using the expansion
\be\label{series2}
   \frac {1}{|1-\rho\, e^{-iu}|^2}
   = \frac{1}{1-\rho^2} \left( 1+ \sum_{k=1}^\infty \rho^{k}
   \left( e^{iku}+ e^{-iku}\right)\right)
\ee
we obtain immediately that
\[
    \int_{-\pi}^\pi  \frac {du}{|1- \rho\, e^{-iu}|^2}
    =  \frac{2\pi}{1-\rho^2} \, .
\]
Moreover, it follows from \eqref{series1}  and \eqref{series2} that
\[
     \int_{-\pi}^\pi  \frac{1}{ \al^2 |e^{iu}-1|^2 +1}
     \  \frac {du}{|1-\rho\, e^{-iu}|^2}
     = \frac{2\pi}{\sqrt{1+4\al^2}(1-\rho^2)}\
     \frac{\b+\rho}{\b-\rho}
\]
with $\b=\b(\al)$ defined in \eqref{beta}. Finally,
\be \label{errna_ar}
  \sigma^2= \frac{\V}{1-\rho^2} \left( 1- \frac{1}{\sqrt{1+4\al^2}}\
  \frac{\b+\rho}{\b-\rho}  \right).
\ee

Notice that Example \ref{ex:iid} is a special case of this one
with $\rho=0$.
}
\end{example}
\smallskip


\begin{example}
{\rm
We call a sequence of random variables  $(\B(t))_{t\in\Z}$
a simplest moving average sequence if it admits a representation
$\B(t)= \xi(t)+\rho\, \xi(t-1)$ where $\xi(t)$ is the same as
in Example $\ref{ex:ar}$. Proceeding as above, we obtain
\[
   \B(t)=  \int_{-\pi}^\pi  \left(1+\rho e^{-iu} \right) e^{itu} \W(du),
   \qquad t\in \Z.
\]
We see that the spectral measure for $\B$ is
\be\label{mu_ma1}
   \mu(du):= \frac {\V |1+ \rho\, e^{-iu}|^2 du}{2\pi}\,.
\ee

The best non-adaptive approximation is given by
\eqref{XgB_star_discr}. By  \eqref{mu_ma1} and \eqref{errna_kin_discr},
the error of non-adaptive approximation is
\begin{eqnarray*}
  \sigma^2 &=&   \frac {\V}{2\pi}
  \int_{-\pi}^\pi  \frac{\al^2 |e^{iu}-1|^2}{ \al^2 |e^{iu}-1|^2 +1}\
  |1+ \rho\, e^{-iu}|^2 du
  \\
    &=&   \frac {\V}{2\pi}
     \int_{-\pi}^\pi  \left[  1- \frac{1}{ \al^2 |e^{iu}-1|^2 +1}\right]\
     |1+ \rho\, e^{-iu}|^2 du
  \\
    &=& \frac {\V}{2\pi} \left[
     \int_{-\pi}^\pi  |1+ \rho\, e^{-iu}|^2 du
     -
      \int_{-\pi}^\pi  \frac{|1+ \rho\, e^{-iu}|^2}
      {\al^2 |e^{iu}-1|^2 +1}\ du  \right].
\end{eqnarray*}
We easily get
\[
   \int_{-\pi}^\pi  |1+ \rho\, e^{-iu}|^2 du
   =  \int_{-\pi}^\pi (1+\rho^2 + \rho(e^{-iu}+e^{iu}))  du
   = 2\pi (1+\rho^2)
\]
and, by using \eqref{series1},
\begin{eqnarray*}
   && \int_{-\pi}^\pi  \frac{|1+ \rho\, e^{-iu}|^2}
   { \al^2 |e^{iu}-1|^2 +1}\  du =  \frac{1}{\sqrt{1+4\al^2}} \times
   \\
   &\times&  \int_{-\pi}^\pi  (1+\rho^2 + \rho(e^{-iu}+e^{iu}))
    \left( 1+ \sum_{k=1}^\infty \b^{-k} \left( e^{iku}
    + e^{-iku}\right)\right) du
   \\
    &=& \frac{2\pi}{\sqrt{1+4\al^2}}
    \left(  1+\rho^2 + \frac{2\rho}{\b}  \right),
\end{eqnarray*}
whereas
\[
  \sigma^2= \V \left( 1+\rho^2 - \frac{1}{\sqrt{1+4\al^2}}
  \left(  1+\rho^2 + \frac{2\rho}{\b}  \right)\right)
\]
with $\b=\b(\al)$ defined in \eqref{beta}.
}
\end{example}


\begin{example} \label{ex:sums}
{\rm Consider the partial sums of a sequence of centered
non-correlated random variables $(\xi(j))_{j\ge 1}$ each having
variance $\V$. Let $\B(0)=0$ and $\B(t):=\sum_{j=1}^t \xi(j)$,
$t\in\N$. Given the spectral representation in the form
$\xi_j = \int_{-\pi}^\pi e^{i t u} \W(du)$, we have
\[
   \B(t)= \int_{-\pi}^\pi \Big( \sum_{j=1}^t  e^{i j u} \Big) \W(du)
   = \int_{-\pi}^\pi \frac{e^{i u}}{e^{i u}-1} \ (e^{i t u}-1) \W(du)
\]
and we obtain the spectral measure
\be \label{mu_sums}
   \mu(du):= \frac {\V du}{2\pi |e^{iu}-1|^2}.
\ee

The best non-adaptive approximation is given by \eqref{XgB_star_discr}.
By \eqref{mu_sums} and \eqref{errna_kin_discr}, the corresponding
approximation error of
\[
  \sigma^2 =  \frac{\V}{2\pi}
  \int_{-\pi}^\pi  \frac{\al^2 \, du}{ \al^2 |e^{iu}-1|^2 +1}\, .
\]
Furthermore, by using expansion \eqref{series1} we have
\be \label{errna_sums}
  \sigma^2 =  \frac{\V \al^2}{2\pi} \cdot \frac{2\pi}{\sqrt{1+4\al^2}}
  =  \frac{\V \al^2}{\sqrt{1+4\al^2}}\, .
\ee
}
\end{example}
\bigskip

We pass now to continuous time examples.


\begin{example}
{\rm
The Ornstein--Uhlenbeck process is a centered Gaussian stationary
process with covariance $\K_\B(t)=e^{-|t|/2}$ and the spectral measure
\be \label{mu_OU}
    \mu(du):= \frac {2 du}{\pi(4u^2+1)}.
\ee

Since we are developing only a linear theory, we do not need
Gaussianity assumption and may call Ornstein--Uhlenbeck any wide
sense stationary process having the mentioned covariance and spectral
measure.

The best non-adaptive approximation is given by \eqref{XgB_star}.
By \eqref{mu_OU} and \eqref{errna_kin}, the error of non-adaptive
approximation is
\[
  \sigma^2 =   \int_\R  \frac{\al^2 u^2}{ \al^2 u^2 +1}\
  \frac{2 du}{\pi(4 u^2+1)} = \frac{\al}{2+\al}\, .
\]
}
\end{example}


\begin{example} \label{ex:fbm}
{\rm
Fractional Brownian motion $(B^H(t))_{t\in\R}$,
$0<H\le 1$,  is a centered Gaussian process with covariance
\[
   \cov (B^H(t_1),B^H(t_2))
   = \frac 12\left( |t_1|^{2H} + |t_2|^{2H} - |t_1-t_2|^{2H}\right).
\]
For any process with this covariance (interesting non-Gaussian
examples of this type are also known, see e.g. Telecom processes
in \cite{KajTaq, RPE}) the spectral measure is
\be \label{mu_fbm}
    \mu(du):= \frac { M_H du}{|u|^{2H+1}}\,
\ee
where $M_H=\tfrac{\Gamma(2H+1) \sin(\pi H)}{2\pi}$\,.

The best non-adaptive approximation is given by \eqref{XgB_star}.
By \eqref{mu_fbm} and  \eqref{errna_kin}, the error of non-adaptive
approximation is
\begin{eqnarray} \nonumber
  \sigma^2 &=&   \int_\R  \frac{\al^2 u^2}{ \al^2 u^2 +1}\  \frac { M_H du}{|u|^{2H+1}}
  =  M_H \al^2 \int_\R \  \frac {|u|^{1-2H} du}{\al^2 u^2 +1 }
  \\ \nonumber
   &=&  2 M_H \al^{2H} \int_0^\infty \  \frac {w^{1-2H} dw}{w^2 +1 }
   = M_H \al^{2H} \int_0^\infty \  \frac {v^{-H} dv}{v +1}
  \\ \label{errna_fbm}
    &=&  M_H \al^{2H} \cdot \frac{\pi}{\sin(\pi H)}
    = \frac{\Gamma(2H+1)\, \al^{2H} }{2}\, .
\end{eqnarray}
This result was obtained in \cite{KabLi}.
}
\end{example}


\begin{example}
{\rm
Consider a centered L\'evy process $(\B(t))_{t\ge 0}$ with finite
variance (this class includes Wiener process and centered Poisson
processes of any constant intensity). Let $\Var \B(1)= \V$. For any
such process the spectral measure is
\be \label{mu_w}
   \mu(du):= \frac {\V du}{2\pi u^2}.
\ee

This is a continuous version of Example $\ref{ex:sums}$, as well
as a special case of  Example \ref{ex:fbm} with $H=\tfrac 12$.
Notice, however, that in the more delicate problems of
{\it adaptive} approximation (that are not studied here) the cases
$H=\tfrac 12$ and $H \not =\tfrac 12$ are totally different.

The best non-adaptive approximation is given by \eqref{XgB_star}.
The calculation of non-adaptive approximation error based
on \eqref{mu_w} and \eqref{errna_kin} is a special case
of \eqref{errna_fbm} with $H=\tfrac 12$, up to a scaling constant
$\V$. We have thus
\be  \label{errna_w}
  \sigma^2= \frac{\al\,\V}{2}\,.
\ee
}
\end{example}
\bigskip

Finally, notice that there is a natural interplay between continuous
time and discrete time approximation. Let $(\B(t))_{t\in\R}$ be a
continuous time wide sense stationary process.
For any small $\delta>0$ consider its discrete time version
$(B_\delta(s))_{s\in\Z}:= (\B(\delta s ))_{s\in\Z}$.
The discrete time counterpart for kinetic energy  $\al^2 \X'(t)^2$
of approximating process is
$\tfrac{\al^2(\X(\delta(s+1))- \X(\delta s))^2}{\delta^2}$,
thus one should consider discrete time approximation with parameter
$\al_\delta:=\tfrac{\al}{\delta}$. Using \eqref{beta}, we see that
\[
  \b_\delta:= \b(\al_\delta) = 1 + \frac{1+o(1)}{\al_\delta}\, ,
  \qquad \textrm{as } \al_\delta\to \infty,
\]
hence
$\b_\delta^{\al_\delta} \to e$ as $\al_\delta\to \infty$.

Therefore, for the best discrete time approximations
\eqref{XgB_star_discr} of $B_\delta$ we have
\begin{eqnarray*}
    \X_\delta(s) &=&
   \frac{1}{\sqrt{1+4\al_\delta^2}}
    \left( \B(s) + \sum_{k=1}^\infty \b_\delta^{-k}
   \left( \B(s+k\delta)+ \B(s-k\delta) \right)\right)
\\
   &=&  \frac{(1+o(1)) \delta}{2\al}
    \left( B_\delta(s) + \sum_{k=1}^\infty [\b_\delta^{\al_\delta}]^{-k\delta/\al}
    \left( \B(s+k\delta)+ \B(s-k\delta) \right)\right)
\\
     &\to& \frac 1{2\al} \int_\R \exp\{-|\tau|/\al\} \B(s+\tau)\, d\tau,
    \qquad   \textrm{as } \delta\to 0,
\end{eqnarray*}
which is the solution of continuous time approximation problem
\eqref{XgB_star}. Similarly, one has the convergence of the optimal
errors of approximation, cf. e.g. \eqref{errna_sums} and
\eqref{errna_w}.



\section{An extension of A.N. Kolmogorov and M.G. Krein theorems
on error-free prediction}

\subsection{Discrete time}
\label{ss:discr}
In this subsection we assume that $(\B(t))_{t\in \Z}$ is a  wide sense
centered stationary sequence and $\mu$ is its spectral measure. Let us represent
$\mu$ as the sum $\mu=\mu_a+\mu_s$ of its absolutely continuous and
singular components. We denote by $f_a$ the density of $\mu_a$ with
respect to the Lebesgue measure.

Consider Problem II and let
\[
  \sigma^2(t):=\inf_{Y\in H_t} \left\{\E|Y-\B(0)|^2+ \E|L Y|^2\right\},
  \qquad t\in \Z,
\]
be the corresponding prediction errors. We also let $\sigma^2(\infty)$
denote the similar quantity with $H_t$ replaced by $H$. It is easy to
see that the sequence  $\sigma^2(t)$ is non-increasing in $t$ and
\[
  \lim_{t\to+\infty}  \sigma^2(t) = \sigma^2(\infty).
\]

For the classical prediction problem, i.e. for $L=0$, by Kolmogorov's
theorem (singularity criterion, see \cite[Chapter II, Theorem 5.4]{Roz})
we have
\[
   \sigma^2(t) = \sigma^2(\infty) = 0 \qquad  \textrm{for all } t\in \Z,
\]
iff
\be \label{logdiv}
   \int_{-\pi}^\pi |\ln f_a(u)| du = \infty.
\ee
In our case, for $L\not=0$, we have $\sigma^2(t)\ge \sigma^2(\infty)>0$
unless $\ell(\cdot)\equiv 0$\ $\mu$-a.s. Therefore, we
state the problem as follows: when $\sigma^2(t)=\sigma^2(\infty)$
holds for a given $t\in\Z$?
In other words: when approximation based on the knowledge of
the process up to time $t$ works as well as the one based on
the knowledge of the whole process?

\begin{thm} \label{t:logdiv}
   If \eqref{logdiv} holds, then we have
   $\sigma^2(t)=\sigma^2(\infty)$ for all $t\in \Z$.
\end{thm}

\begin{proof}
If \eqref{logdiv} holds, then for all $t$ we have $H_t=H$, see e.g.
\cite[Chapter XII, Section 4]{Doob1} or
\cite[Chapter II, Section 2]{Roz}. Therefore, by Theorem \ref{t:nonad}
\[
  \sigma^2(t)=\sigma^2(\infty)
  =\int_{-\pi}^{\pi} \frac{|\ell(u)|^2}{1+|\ell(u)|^2}\, \mu(du).
\]
\end{proof}

\begin{thm} \label{t:logconv}
   If the process $\B$  is such that the density $f_a$ satisfies
\be \label{logconv}
   \int_{-\pi}^\pi |\ln f_a(u)| du < \infty,
\ee
then for every fixed $t\in \Z$ we have $\sigma^2(t)=\sigma^2(\infty)$ iff
the function 
$\frac{1}{1+|\ell(u)|^2}$
is a trigonometric polynomial of degree not exceeding $t$, i.e.
\[
  \frac{1}{1+|\ell(u)|^2}= \sum_{|j|\le t} b_{j}\, e^{iju}
\]
Lebesgue-a.e. with some coefficients $b_j\in \C$.

In particular, if $t<0$, then $\sigma^2(t)<\sigma^2(\infty)$;
equality $\sigma^2(0)=\sigma^2(\infty)$ holds iff $|\ell(\cdot)|$ is
a constant Lebesgue-a.e.
\end{thm}

\begin{proof}
The analytic form for prediction error is
\[
  \sigma^2(t) =
  \inf_{\psi \in \LL(t)} \int_{-\pi}^\pi
  \left\{   |\psi(u)-1|^2 + |\ell(u)\psi(u)|^2 \right\} \mu(du),
\]
where $\LL(t)= \span{e^{isu}, s\le t}$ in $\LL$.

The solution $\psi_t(u)$ of our problem is unique by Proposition \ref{p:uniq};
we know from \eqref{sol_na_anal}  that for $t=\infty$ the solution is
$\psi_\infty(u)=\tfrac{1}{1+|\ell(u)|^2}$. It follows that
$\sigma^2(t)=\sigma^2(\infty)$ iff  $\psi_t=\psi_\infty$,
i.e. iff $\psi_\infty\in\LL(t)$. The latter is equivalent to
the existence of the trigonometric polynomials
\[
   \vartheta_k(u):=\sum_{j\le t} a_{j,k}\, e^{iju}
\]
such that
\[
  \lim_{k\to\infty} \int_{-\pi}^\pi
     \left| \vartheta_k(u)- \frac{1}{1+|\ell(u)|^2}\right|^2 \mu(du) =0.
\]
It follows that
\begin{eqnarray} \nonumber
  && \lim_{k\to\infty} \int_{-\pi}^\pi
     \left| \overline{\vartheta_k(u)}- \frac{1}{1+|\ell(u)|^2}\right|^2 f_a(u)\, du
  \\ \label{psikconv}
  &=& \lim_{k\to\infty} \int_{-\pi}^\pi
   \left|\vartheta_k(u)-\frac{1}{1+|\ell(u)|^2}\right|^2 f_a(u)\,du
   =0.
\end{eqnarray}

Due to assumption \eqref{logconv}
the density $f_a$ admits a representation
\[
   f_a(u)=|g_*(e^{iu})|^2
\]
where $g_*(e^{iu}), u\in[-\pi,\pi),$ is the boundary value of the function
\[
  g(z) := \exp\left\{ \frac{1}{4\pi} \int_{-\pi}^\pi \ln f_a(u)\
          \frac{e^{iu}+z}{e^{iu}-z}\ du \right\}, \qquad
          |z|<1,
\]
which is an analytic function in the unit disc $\D:=\{z\in\C, |z|<1\}$.
In other words, $g$ is an outer function from the Hardy class $\HH_2(\D)$;
we refer to \cite[Chapter 17]{Ru} for the facts and definitions mentioned in this subsection
concerning $\HH_2(\D)$, outer and inner functions.
Notice that $\tfrac{1}{g}$ also is an outer analytic function.

Rewrite \eqref{psikconv} as
\[
   \lim_{k\to\infty} \int_{-\pi}^\pi \left| \overline{\vartheta_k(u)} g_*(e^{iu})
   - \frac{g_*(e^{iu})}{1+|\ell(u)|^2} \right|^2 du = 0.
\]
Assume for a while that $t\le 0$. Then
$\overline{\vartheta_k} g_*$ also is the boundary value of a function from $\HH_2(\D)$.
Since the class of such boundary functions is closed in $L_2$ (with respect to Lebesgue
measure), this implies that
\[
    h_*(e^{iu}):= \frac{g_*(e^{iu})}{1+|\ell(u)|^2}
 \]
is the boundary value of a function $h\in \HH_2(\D)$.
Moreover, we have a power series representation
\be \label{hpower}
    h(z)=\sum_{j\ge - t} h_j \, z^j, \qquad z\in \D.
\ee
Let us denote
\[
   A_1(u):=\frac{1}{1+|\ell(u)|^2} = h_*(e^{iu}) \cdot \frac{1}{g_*(e^{iu})}\, ,
   \qquad u\in [-\pi,\pi).
\]
The function $e^{iu}\mapsto A_1(u)$ admits an analytic continuation
from the unit circle to $\D$ given by
$A(z)=h(z)\cdot \tfrac{1}{g(z)}$.
Notice that in the power series representation
\[
   A(z) = \sum_{j=0}^\infty a_j\, z^j, \qquad z\in\D,
\]
the terms with $j<-t$ vanish due to \eqref{hpower}.
We have therefore
\[
   A(z) = \sum_{j\ge -t}^\infty a_j \,z^j, \qquad z\in\D.
\]

Let us prove that $A(\cdot)$ is bounded on $\D$. Write the
factorization $h=M_h\cdot Q_h$ where $M_h$ is an inner function and $Q_h$ is an outer function.
Then $A= M_h\cdot \frac{Q_h}{g}$. The function $M_h$ is bounded  on $\D$ by the
definition of an inner function while $\frac{Q_h}{g}$ is an outer function
with bounded boundary values, because Lebesgue-a.e.
\[
   \left|\left(\tfrac{Q_h}{g}\right)_*(e^{iu})\right|
   =\left|\frac{A_1(u)}{(M_h)_*(e^{iu})}\right|
   = \left| A_1(u)\right|\le 1.
\]
Since for outer functions the boundedness on the boundary implies,
via Poisson kernel representation, the boundedness on $\D$, we see that
the factor $\frac{Q_h}{g}$ is also bounded on $\D$. We conclude that
$A(\cdot)$ is bounded on $\D$.

For each $r\in (0,1)$ consider the function
\[
   A_r(u):=A(re^{iu})= \sum_{j\ge -t}^\infty a_j\, r^j\, e^{iju}\, ,
   \qquad u\in [-\pi,\pi).
\]

Since $A$ is bounded, the family $\{A_r\}_{0<r<1}$ is uniformly bounded.
Since $A_r\to A_1$ Lebesgue-a.e., as $r\nearrow 1$, the convergence also holds
in $L_2$. In particular, all Fourier coefficients converge and we have
\[
   A_1(u) = \frac{1}{1+|\ell(u)|^2} = \sum_{j\ge -t} a_j \,e^{iju}.
\]
Since the left hand side is real, for $t<0$ the latter representation
is impossible. For $t=0$ it is only possible when both sides are equal
(Lebesgue-a.e.) to the constant $a_0$.

For $t>0$ the same reasonings give a representation
\[
   \frac{e^{itu}}{1+|\ell(u)|^2} = \sum_{j\ge 0} a_j e^{iju}
\]
which implies that
\[
   \frac{1}{1+|\ell(u)|^2} = \sum_{j=-t}^{t} a_{j+t} e^{iju}
\]
is a trigonometric polynomial of degree not exceeding $t$.
\medskip

The converse assertion is obvious: if for $t\ge 0$
we have a representation
\[
   \psi_\infty(u) = \frac{1}{1+|\ell(u)|^2} = \sum_{j=-t}^{t} b_j e^{iju},
\]
then $\psi_\infty \in \LL(t)$ by the definition of $\LL(t)$.
\end{proof}

Theorems \ref{t:logdiv} and \ref{t:logconv} immediately yield the following
final result.

\begin{thm} \label{t:log}
Let $\B$ be a discrete time, wide sense stationary process. Let $L$ be a
linear filter with frequency characteristic $\ell(\cdot)$. Then for
every fixed $t\in \Z$
the equality $\sigma^2(t)=\sigma^2(\infty)$ holds iff either
\eqref{logdiv} holds, or \eqref{logconv} holds and
$\frac{1}{1+|\ell(u)|^2}$ is a trigonometric polynomial of degree
not exceeding $t$.
\end{thm}

\subsection{Continuous time}

In this subsection we assume that $(\B(t))_{t\in \R}$ is a continuous time,
mean square continuous, wide sense stationary process and $\mu$ is its spectral
measure. As before, we represent $\mu$ as the sum $\mu=\mu_a+\mu_s$ of
its absolutely continuous and singular components, denote $f_a$ the density
of $\mu_a$ with respect to Lebesgue measure and let
\[
  \sigma^2(t) :=\inf_{Y\in H_t} \left\{ \E|Y-\B(0)|^2+ \E|L Y|^2 \right\},
  \qquad t\in \R,
\]
denote the corresponding prediction errors. We also let $\sigma^2(\infty)$
denote the similar quantity with $H_t$ replaced by $H$.

The statement analogous to Theorem \ref{t:log} is as follows.

\begin{thm} \label{t:log_c}
Let $\B$ be a continuous time, mean square continuous, wide sense stationary process.
Let $L$ be a linear filter with frequency characteristic $\ell(\cdot)$.
Then for every fixed $t\in \R$ the equality
\[
  \sigma^2(t)=\sigma^2(\infty)
  = \int \frac{|\ell(u)|^2}{1+ |\ell(u)|^2} \, \mu(du)
\]
holds iff either

\noindent(a)
\be \label{logdiv_c}
   \int_{-\infty}^{\infty} \frac{|\ln f_a(u)|}{1+u^2}\,  du = \infty
\ee
holds or

\noindent(b)
\be \label{logconv_c}
   \int_{-\infty}^{\infty}  \frac{|\ln f_a(u)|}{1+u^2}\,  du < \infty
\ee
holds, $t>0$ and $\frac{1}{1+|\ell(u)|^2}$ is a restriction (to $\R$) of
an entire analytic function of exponential type not exceeding $t$,
or

\noindent(c) inequality \eqref{logconv_c} holds, $t=0$, and
$|\ell(\cdot)|$ is Lebesgue a.e. equal to a constant.
\end{thm}

\begin{proof}
If \eqref{logdiv_c} holds, then by M.G. Krein singularity criterion,
we have $H_t=H$ for all $t\in \R$, see e.g.
\cite[Chapter XII, Section 4]{Doob1} or \cite[Chapter II, Section 2]{Roz},
and the assertion a) of the theorem follows.

Let $\Pi:=\{z\in\C: \Im(z)>0\}$ denote the upper half-plane. If \eqref{logconv_c} holds, then we
have a representation $f_a(u)=|g_*(u)|^2$, where $g_*(u)$ is the boundary value
of the function
\[
   g(z) := \exp\left\{ \frac{1}{2\pi\, i} \int_{-\infty}^\infty
   \ln f_a(u) \frac{1+z u}{z-u}\, \frac{du}{1+u^2} \right\}.
\]
Therefore $g$ is an outer function from Hardy class $\HH_2(\Pi)$;
we refer to \cite[Chapter 8]{Hof} for the facts and definitions mentioned in this subsection
concerning Hardy classes on $\Pi$ and related outer and inner functions.

Let $t\le 0$. The same arguments as those given in the proof of Theorem \ref{t:logconv}
show that
\[
    h_*(u):= \frac{g_*(u)}{1+|\ell(u)|^2}
 \]
is the boundary value of a function $h\in \HH_2(\Pi)$.

The function
\be \label{idenR}
   A_*(u):=\frac{1}{1+|\ell(u)|^2} = h_*(u) \cdot \frac{1}{g_*(u)}\, ,
   \qquad u\in \R,
\ee
admits an analytic continuation from the real line to $\Pi$ given by
$A(z)=h(z)\cdot \tfrac{1}{g(z)}$.

Let us prove that $A(\cdot)$ is bounded on $\Pi$. Write the factorization
$h=M_h\cdot Q_h$ where $M_h$ is an inner function and $Q_h$ is an outer function.
Then $A= M_h\cdot \frac{Q_h}{g}$. The function $M_h$ is bounded  on $\Pi$ by the
definition of an inner function while $\frac{Q_h}{g}$ is an outer function
with bounded boundary values, because Lebesgue-a.e.
\[
   \left|\left(\tfrac{Q_h}{g}\right)_*(u)\right|
   =\left|\frac{A_*(u)}{(M_h)_*(u)}\right|
   = \left| A_*(u)\right|\le 1, \qquad u\in \R.
\]
Since for outer functions the boundedness on the boundary implies,
via Poisson kernel representation, the boundedness on $\Pi$, we see that
the factor $\frac{Q_h}{g}$ is also bounded on $\Pi$. We conclude that
$A(\cdot)$ is bounded on $\Pi$. In other words, $A\in \HH_\infty(\Pi)$.

Furthermore, the function $A$ admits an analytic reflection to the lower
half-plane $\Pi_-:=\{z\in\C: \overline{z}\in \Pi\}$ by letting
$A_-(z):=\overline{A(\overline{z})}$. This reflection agrees with $A$ on the
real line because the boundary values $A_*(u),u\in\R,$ are real.

Consider now an auxiliary function $\SS_*(u):=\tfrac{\sin u}{u}\ e^{iu}$ on $\R$
and its analytic continuation $\SS(z):=\tfrac{\sin z}{z}\ e^{iz} \in \HH_2(\Pi)$.
Then $\SS\cdot A\in \HH_2(\Pi)$ and the corresponding boundary function
$\SS_*\cdot A_*\in L_2(\R)$. According to the Fourier representation of the elements of
$\HH_2(\Pi)$ and that of their boundary values  we have
\begin{eqnarray}
\nonumber
   &&\SS_*(u) A_*(u) =  \frac{\sin u}{u}\ e^{iu} A_*(u)
\\   \label{Hu}
   &=& \int_0^\infty e^{ixu} q(x)\, dx, \qquad u\in \R, \textrm{ Lebesgue-a.e.,}
\\ \nonumber
   && \SS(z) A(z) =  \frac{\sin z}{z}\ e^{iz} A(z)
\\ \label{Hz}
   &=& \int_0^\infty e^{ixz} q(x)\, dx, \qquad z\in \Pi,
\end{eqnarray}
with some $q(\cdot)\in L_2(\R_+)$.
We may rewrite \eqref{Hu} as
\[
  \frac{\sin u}{u}\,  A_*(u) = \int_{-1}^\infty e^{iyu} q(y+1)\,dy.
\]
Moreover, since the left hand side is real, its Fourier transform is symmetric.
Therefore, $q(\cdot+1)$ must vanish on $[1,\infty)$, i.e.
 $q(\cdot)$ must vanish on $[2,\infty)$. Thus \eqref{Hz} writes as
\be \label{SAQ}
  \SS(z)\, A(z) = \int_{0}^2 e^{ixu} q(x)\,dx:=Q(z)
\ee
and $Q$ is an entire function.

Consider the holomorphic function
\[
  V(z):=\frac{Q(z)}{\SS(z)}= \frac{z e^{-iz} Q(z)}{\sin z}.
\]
Since $V=A$ on $\Pi$ and $V=A_*$ on $\R$, we see that $V$ is bounded on $\Pi\cup\R$.

Proceeding in the same way with the lower half-plane $\Pi_-$ instead of $\Pi$, we find "another"
holomorphic function $V_-$ such that $V_-=A_-$ on $\Pi_-$ and $V_-=A_*=V$ on $\R$. The latter
equality yields $V=V_-$ on $\C$; moreover, $V$ is bounded on $\C$. Hence $V$ is a constant and
$A_*=V$ is a constant, too, as required in the assertion c) of our theorem.
\medskip

If \eqref{logconv_c} holds and $t>0$, the same reasoning leads to
a representation analogue to \eqref{idenR}, namely,
\[
   A_*(u):=\frac{1}{1+|\ell(u)|^2} = e^{-itu} \AA_*(u)\, ,
   \qquad u\in \R,
\]
where $\AA_*$ is the boundary function of some $\AA\in\HH_\infty(\Pi)$.

Using that $\SS\cdot\AA\in \HH_2(\Pi)$
and proceeding as before, we have a representation analogue to \eqref{SAQ}
\[
  \SS(z)\, \AA(z) = \int_{0}^{2(t+1)} e^{izu} q(x)\,dx:=Q(z)
\]
and $Q$ is an entire function. Let
\[
  V(z)= \frac{Q(z)e^{-itz}}{\SS(z)}= \AA(z) e^{-itz}.
\]
Then
$V=A_*$ on $\R$ and we have
\[
  |V(z)| \le e^{t|\Im(z)|} ||A||_\infty\, , \qquad z\in \Pi.
\]

Proceeding in the same way with the lower half-plane $\Pi_-$ instead of $\Pi$, we find "another"
holomorphic function $V_-$ such that
\[
  |V_-(z)| \le e^{t|\Im(z)|} ||A_-||_\infty\ , \qquad z\in \Pi_-,
\]
and $V_-=A_*=V$ on $\R$ . The latter equality yields $V=V_-$ on $\C$. We conclude that
\[
  |V(z)| \le e^{t|\Im(z)|} ||A||_\infty\, , \qquad z\in \C,
\]
i.e. $V$ is an entire function of exponential type not exceeding $t$,
as required in the assertion b) of our theorem.
\medskip

If \eqref{logconv_c} holds and $t<0$, we still see from the previous reasoning that
$\psi_\infty$ must be a constant.
Due to M.G. Krein's regularity criterion  (see \cite[Chapter III, Theorem 2.4]{Roz}),
we know that under \eqref{logconv_c}
constants do not belong to $\LL(t)$ with $t<0$. Hence the equality
$\sigma^2(t)=\sigma^2(\infty)$ is not possible for $t<0$.
\medskip

Now we prove the sufficiency. Assume that $t>0$ and that
$A_*(u):=\frac{1}{1+|\ell(u)|^2}$ is a restriction (to $\R$) of
an entire analytic function $A(\cdot)$ of exponential type not exceeding $t$.

Write
\[
  A_*(u)=A_*(0)+ u \cdot \frac{A_*(u)-A_*(0)}{u}:= A_*(0)+u \cdot \widetilde{A_*}(u).
\]
It is  sufficient to show that the function $u\mapsto u \cdot \widetilde{A_*}(u)$ belongs to $H_t$.
Furthermore, since we have
\[
   \frac{1-e^{-i\delta u}}{i\delta} \widetilde{A_*}(u) \to  u \cdot \widetilde{A_*}(u),
   \qquad \textrm{as } \delta\to 0,
\]
in $L_2(\R,\mu)$, it is sufficient to show that $\widetilde{A_*}\in H_t$.

Notice that $\widetilde{A_*}$ belongs to $L_2(\R^1)$ w.r.t. Lebesgue measure and is a restriction
of the analytic function of exponential type not exceeding $t$ given by $\widetilde{A}(z):=\frac{A(z)-A(0)}{z}$.
Hence, by Paley--Wiener theorem  (see \cite[Chapter IV]{Akh} we have a representation
\[
   \widetilde{A_*}(u) = \widetilde{A}(u) = \int_{-t}^t a(\tau) e^{i\tau u} d\tau
\]
with $a(\cdot)\in L_2[-t,t]\subset L_1[-t,t]$.
Since exponentials $u \mapsto e^{i\tau u}$ belong to $H_t$ as $\tau\le t$, it follows that
$\widetilde{A_*}\in H_t$.

This concludes the proof of our theorem.
\end{proof}

\begin{rem}
{\rm In 1923, S.N. Bernstein introduced a class of entire functions of exponential type
not exceeding $t$ and bounded on the real line, cf. \cite{Ber} or \cite[Chapter 4, Section 83]{Akh}.
The functions that appear in Theorem \ref{t:log_c} belong to this class.
}
\end{rem}


\subsection{An extension of regularity criterion}

Consider again the discrete time case. We handle  wide sense stationary sequences
and use the notation from Subsection \ref{ss:discr}. Let
$\sigma_0^2(t)$ be the prediction error for $\B(0)$ given the past $H_{t}$
in the classical prediction problem (with $L=0$).
The sequence $\B$ is called {\it regular}, if we have
\[
   \lim_{t\to -\infty} \sigma_0^2(t) = \E |\B(0)|^2.
\]
By the classical Kolmogorov regularity criterion, see
\cite[Chapter II, Theorem 5.1]{Roz}, a   wide sense stationary sequence $\B$ is regular
iff its spectral measure is absolutely continuous and \eqref{logconv} is
verified.

The following result provides an extension of this assertion to our
settings.

\begin{thm}
Let $\B$ be a wide sense stationary sequence. Let $L$ be a linear filter with
frequency characteristic $\ell(\cdot)$. If \eqref{logconv} holds, then
\[
      \lim_{t\to-\infty} \sigma^2(t) = \E |\B(0)|^2
      - \int_{-\pi}^\pi \frac{\mu_s(du)}{1+|\ell(u)|^2}\ .
\]
\end{thm}

\begin{proof}
Consider first the sequences with absolutely continuous spectral measure.
In that case, by Kolmogorov's criterion, $\B$ is regular. Hence,
\[
   \lim_{t\to-\infty} \sigma^2(t)
   \ge \lim_{t\to-\infty} \inf_{Y\in H_t} \E|Y-\B(0)|^2
   = \E|\B(0)|^2.
\]
On the other hand, since for each $t\in\Z$ it is true $Y=0\in H_{t}$,
we have $\sigma^2(t) \le \E|\B(0)|^2$
and the theorem is proved in the form
\be \label{regabsc}
   \lim_{t\to-\infty} \sigma^2(t) = \E|\B(0)|^2.
\ee

In the general case,
by  \cite[Chapter II, Theorem 2.2]{Roz} our sequence splits into a sum
$B=B^{(a)}+B^{(s)}$ of mutually orthogonal wide sense stationary processes such that
the regular part $B^{(a)}$ has the spectral measure $\mu_a$, the singular part
$B^{(s)}$ has the spectral measure $\mu_s$ and
the corresponding spaces
$H_{t}^{(a)}:=\span{B^{(a)}(v),v\le t}$,
$H_{t}^{(s)}:=\span{B^{(s)}(v), v\le t}$ are not only orthogonal but
also satisfy subordination inclusions $H_{t}^{(a)}\subset H_t$,
$H_{t}^{(s)}\subset H_t$. The latter yield
$H_t=H_{t}^{(a)}\oplus H_{t}^{(s)}$.

Moreover, for any $\xi \in H_t$ representation \eqref{L} implies
\begin{eqnarray*}
   \E |L\xi|^2 &=& \int_{-\pi}^{\pi} |\ell(u)\phi_\xi(u)|^2\mu(du)
\\
   &=&  \int_{-\pi}^{\pi} |\ell(u)\phi_\xi(u)|^2\mu_a(du)
      + \int_{-\pi}^{\pi} |\ell(u)\phi_\xi(u)|^2\mu_s(du)
\\
    &=& \E |L\xi^{(a)}|^2 + \E |L\xi^{(s)}|^2,
\end{eqnarray*}
where $\xi^{(a)}, \xi^{(s)}$ denote the projections of $\xi$
onto $H_{t}^{(a)}$ and $H_{t}^{(s)}$, respectively.

Therefore, for any $\xi\in H_t$ we have
\begin{eqnarray*}
   &&\E |\xi-\B(0)|^2+ \E|L\xi|^2
\\
   &=&  \E |\xi^{(a)}-B^{(a)}(0)|^2+ \E |\xi^{(s)}-B^{(s)}(0)|^2
\\
&&   + \E|L \xi^{(a)}|^2 + \E|L \xi^{(s)}|^2.
\end{eqnarray*}

In this situation, the minimization problem splits into two independent
ones and is solved by $\xi_t= \xi_t^{(a)}+\xi_t^{(s)}$, where
$\xi_t^{(a)},\xi_t^{(s)}$ are the solutions for the processes
$ B^{(a)}, B^{(s)}$, respectively.

For the extended prediction errors
we obtain, by using Theorem \ref{t:logdiv},
\be \label{sigma_as}
  \sigma^2(t)= \sigma^{(a)2}(t) + \sigma^{(s)2}(t)
  =  \sigma^{(a)2}(t) + \sigma^{(s)2}(\infty).
\ee

By applying consequently \eqref{sigma_as} and \eqref{regabsc}, we have
\begin{eqnarray*}
   \lim_{t\to-\infty} \sigma^2(t)&=&
    \lim_{t\to-\infty} \sigma^{(a)2}(t)+ \sigma^{(s)2}(\infty)
\\
   &=&  \E|B^{(a)}(0)|^2
       + \int_{-\pi}^\pi \frac{|\ell(u)|^2 \mu_s(du)}{1+|\ell(u)|^2}
\\
   &=&  \E|\B(0)|^2 - \E|B^{(s)}(0)|^2
        + \int_{-\pi}^\pi \frac{|\ell(u)|^2\mu_s(du)}{1+|\ell(u)|^2}
\\
   &=&  \E|\B(0)|^2 + \int_{-\pi}^\pi
        \left(\frac{|\ell(u)|^2}{1+|\ell(u)|^2}-1\right) \mu_s(du)
\\
   &=&  \E|\B(0)|^2 -
   \int_{-\pi}^\pi \frac{ \mu_s(du)}{1+|\ell(u)|^2}\, ,
\end{eqnarray*}
as claimed.
\end{proof}


\section{Interpolation}

Consider the simplest case of interpolation problem (our Problem III)
in discrete time. Let $(B(t))_{t\in \Z}$ be a  wide sense stationary sequence having
spectral density $f$ and let $L$ be a linear filter with frequency
characteristic $\ell(\cdot)$. Consider the extremal problem
\[
    \E |Y-\B(0)|^2+ \E|LY|^2 \to \min, \qquad Y\in \HTO_1.
\]
Recall that $\HTO_1=\span{\B(s),|s|\ge 1}$.
Let
\[
   \sigma_{\textrm{int}}^2   = \inf_{Y\in \HTO_1} \left(\E |Y-\B(0)|^2+ \E|LY|^2\right)
\]
denote the interpolation error.

The classical case of this problem, i.e. $L=0$, was considered by A.N. Kolmogorov
 \cite{Kolm}. He proved that precise extrapolation with $ \sigma_{\textrm{int}}^2=0$ is possible iff
 \be \label{interdiv}
     \int_{-\pi}^\pi \frac{du}{f(u)} =\infty.
 \ee
If the integral in \eqref{interdiv} is convergent, then
\[
    \sigma_{\textrm{int}}^2 = 4 \pi^2 \left( \int_{-\pi}^\pi \frac{du}{f(u)}\right)^{-1}.
\]
We extend this result to the case of general $L$ as follows.

\begin{thm}
If \eqref{interdiv} holds, then
\[
  \sigma_{\textrm{int}}^2
  =   \int_{-\pi}^\pi \frac{|\ell(u)|^2}{1+|\ell(u)|^2} \ f(u)\, du.
\]
Otherwise,
\begin{eqnarray*}
    \sigma_{\textrm{int}}^2
  &=&   \int_{-\pi}^\pi \frac{|\ell(u)|^2 f(u) du}{1+|\ell(u)|^2}
\\
  &&  + \left( \int_{-\pi}^\pi \frac{du}{1+|\ell(u)|^2} \right)^2
      \left( \int_{-\pi}^\pi \frac{du}{f(u)(1+|\ell(u)|^2)} \right)^{-1}.
\end{eqnarray*}
\end{thm}

\begin{proof}
If \eqref{interdiv} holds, then by Kolmogorov's theorem we have
$\B(0)\in \HTO_1$, thus $\HTO_1=H$ and by \eqref{errna}
\begin{eqnarray*}
   \sigma_{\textrm{int}}^2
   &=&  \inf_{Y\in H} \left( \E |Y-\B(0)|^2+ \E|LY|^2\right)
\\
   &=&    \int_{-\pi}^\pi \frac{|\ell(u)|^2}{1+|\ell(u)|^2} \ f(u)\, du,
\end{eqnarray*}
proving the first assertion of the theorem.

Assume now that
 \be \label{interconv}
     \int_{-\pi}^\pi \frac{du}{f(u)} <\infty.
 \ee
Let us define a function $\phi$ on $[-\pi,\pi)$
by the relations
\begin{eqnarray*}
  \phi(u) &:=& \frac {c+f(u)}{f(u)(1+|\ell(u)|^2)}\ ,
\\
   c &=&  - \left( \int_{-\pi}^\pi \frac{du}{1+|\ell(u)|^2} \right)
            \left( \int_{-\pi}^\pi \frac{du}{f(u)(1+|\ell(u)|^2)} \right)^{-1}.
\end{eqnarray*}
We will prove that the random variable
\[
   \xi = \int_{-\pi}^\pi \phi(u)\, \W(du)
\]
solves our interpolation problem. To this aim, according to Propositions \ref{p:exist},
\ref{p:uniq}, and \ref{p:euler}, it is sufficient to prove that $\xi\in \DD(L)$, that
$\xi\in\HTO_1$, and that $\xi$ satisfies equations \eqref{euler} of Proposition \ref{p:euler}.

First, we have
\begin{eqnarray*}
  \E|L\xi|^2 &=& \int_{-\pi}^\pi |\phi(u)\ell(u)|^2 f(u) du
\\
    &=& \int_{-\pi}^\pi   \frac {(c+f(u))^2}{f(u)}\ \frac{|\ell(u)|^2 }{(1+|\ell(u)|^2)^2} \ du
\\
    &\le&   \frac{1}{4} \int_{-\pi}^\pi \frac {(c+f(u))^2}{f(u)} \, du
\\
   &=&    \frac{c^2}{4} \int_{-\pi}^\pi \frac {du}{f(u)} +  \pi\,c
          +  \frac{1}{4} \int_{-\pi}^\pi f(u) \, du <\infty,
\end{eqnarray*}
whence  $\xi\in \DD(L)$.

Second, we show that $\xi\in\HTO_1$. Consider an orthogonal decomposition
$\xi=\eta+\eta^{\bot}$ with $\eta\in\HTO_1$, $\eta^{\bot}\in(\HTO_1)^{\bot}$.
We also have the corresponding analytic decomposition
$\phi=\psi+\psi^{\bot}$ with
$\psi\in \LLO :=\span{e^{isu},|s|\ge 1}$
and $\psi^{\bot}\in (\LLO)^{\bot}$.

Let us show that any $h\in \LLO$ satisfies equation
\be \label{h0}
   \int_{-\pi}^\pi h(u) \, du  = 0.
\ee
Indeed under assumption \eqref{interconv} the linear functional
$h\mapsto  \int_{-\pi}^\pi h(u) \, du$ is bounded and continuous
on $\LL$ because by H\"older's inequality
\begin{eqnarray*}
  \left|  \int_{-\pi}^\pi h(u) \, du \right|^2
  &\le& \int_{-\pi}^\pi \frac{du}{f(u)}\ \int_{-\pi}^\pi |h(u)|^2 f(u) \, du
\\
   &=& \int_{-\pi}^\pi \frac{du}{f(u)} \ \, ||h||_\LL^2 .
\end{eqnarray*}
Therefore, equality \eqref{h0} which is true for every exponent from the set
$\{h(u)=e^{isu}, |s|\ge 1\}$, extends  to their span $\LLO$.
In particular, we obtain
\[
  \int_{-\pi}^\pi \psi(u) \, du  = 0.
\]
Since the constant $c$ in the definition of $\phi$ was chosen so that
 \be \label{phi0}
    \int_{-\pi}^\pi \phi(u) \, du  = 0,
\ee
we obtain
\be \label{psi0}
  \int_{-\pi}^\pi \psi^{\bot}(u) \, du  =  \int_{-\pi}^\pi (\phi-\psi)(u)\, du = 0.
\ee
On the other hand,  we have $\psi^{\bot} f \in L_1[-\pi,\pi)$ and
$\psi^{\bot}\in (\LLO)^{\bot}$. The latter means that
\[
   \int_{-\pi}^\pi \psi^{\bot}(u) e^{isu} f(u) \, du =0, \qquad s\in \Z, s\not= 0.
\]
Therefore $\psi^{\bot} f$ is a constant, say,  $\psi^{\bot} f=a$. By plugging
$\psi^{\bot}= \frac{a}{f}$ into \eqref{psi0} we obtain $a=0$.
It follows that $\psi^{\bot}=0$ and $\phi=\psi\in \LLO$, which is equivalent to
$\xi\in \HTO_1$.

It remains to check that $\xi$ satisfies equations \eqref{euler}. The analytical
form of these equations is
\be \label{euler_inter}
    \int_{-\pi}^\pi \left[ \phi(u)-1+|\ell(u)|^2 \phi(u) \right] h(u) f(u) du =0,
    \qquad h\in\LLO.
\ee
By the definition of $\phi$ we have
\[
   \phi(u)-1+|\ell(u)|^2 \phi(u) =\frac{c}{f(u)}.
\]
Therefore, \eqref{euler_inter} reduces to
\[
   \int_{-\pi}^\pi h(u) \, du  = 0,  \qquad h\in\LLO.
\]
The latter was already verified in \eqref{h0}, and we have proved that $\xi$ is a solution of
interpolation problem. Now the direct computation using the definitions of $\phi$ and $c$,
as well as \eqref{phi0}, shows
\begin{eqnarray*}
\sigma_{\textrm{int}}^2
&=&  \int_{-\pi}^\pi  \left[(\phi(u)-1)^2+|\ell(u)|^2 \phi(u)^2 \right] f(u)\, du
\\
&=&  \int_{-\pi}^\pi  \left[\phi(u)^2(1+|\ell(u)|^2)  + (1-2\phi(u)) \right] f(u)\, du
\\
&=&  \int_{-\pi}^\pi  \left[\phi(u)(c+f(u))  + (1-2\phi(u)) f(u) \right] du
\\
&=&  \int_{-\pi}^\pi  (1-\phi(u)) f(u) du
= \int_{-\pi}^\pi  \frac{ |\ell(u)|^2f(u)- c }{1+|\ell(u)|^2}\, du
\\
&=& \int_{-\pi}^\pi  \frac{ |\ell(u)|^2f(u)}{1+|\ell(u)|^2}\, du
\\
 && + \left( \int_{-\pi}^\pi \frac{du}{1+|\ell(u)|^2} \right)^2
      \left( \int_{-\pi}^\pi \frac{du}{f(u)(1+|\ell(u)|^2)} \right)^{-1},
\end{eqnarray*}
 as claimed in the theorem's assertion.
\end{proof}

\begin{rem} {\rm
   The first term in $\sigma_{\textrm{int}}^2$ is just the optimal error \eqref{errna} in the
   easier optimization problem with $Y\in H$. The second term is the price we must pay
   for optimization over the smaller space $\HTO_1$ instead of $H$.
}
\end{rem}

\section{Proofs for abstract Hilbert space setting}
\label{s:HSproofs}

\begin{proof}[ of Proposition \ref{p:exist}]
Not that the set $H_0 \bigcap \DD(L)$ is non-empty, since it contains zero.
Let
\[
   \sigma^2:= \inf_{y\in H_0 \bigcap \DD(L)} G(y).
\]
There exists a sequence $(\xi_n)_{n\in \N}$ in $H_0 \bigcap \DD(L)$ such that
\[
  \lim_{n\to\infty} G(\xi_n) =\sigma^2.
\]
Clearly, for all  $m,n\in \N$ we have $(\xi_m+\xi_n)/2\in H_0 \bigcap \DD(L)$
and by convexity of $||\cdot||^2$,
\begin{eqnarray*}
   \left\|\frac{\xi_m+\xi_n}{2} - x\right\|^2
   &\le& \frac 12 \ \left( ||\xi_m-x||^2 +  ||\xi_n-x||^2   \right),
\\
    \left\| L\left( \frac{\xi_m+\xi_n}{2} \right) \right\|^2
    &=& \frac 14 \ \left\| L \xi_m +L \xi_n \right\|^2
\\
    &\le& \frac 12 \ \left( ||L \xi_m||^2 +  ||L \xi_n||^2   \right).
\end{eqnarray*}
It follows that
\[
   \limsup_{m,n\to\infty} G\left(  \frac{\xi_m+\xi_n}{2} \right)
   \le \frac 12 \left( \lim_{m\to\infty} G(\xi_m) + \lim_{n\to\infty} G(\xi_n)\right)
   =\sigma^2.
\]
Hence,
\[
   \lim_{m,n\to\infty} G\left(\frac{\xi_m+\xi_n}{2} \right) =\sigma^2.
\]
The parallelogram identity
\be \label{para}
     2||f||^2 +2 ||g||^2 = ||f+g||^2+||f-g||^2
\ee
yields
\begin{eqnarray*}
     &&  2 \left( G(\xi_m)+G(\xi_n)\right)
\\
    &=& 2\left(||\xi_m-x||^2+||L\xi_m||^2\right)
       + 2\left(||\xi_n-x||^2+||L \xi_n||^2\right)
\\
   &=& ||\xi_m+\xi_n-2x||^2 + ||\xi_m-\xi_n||^2
      + ||L(\xi_m+\xi_n)||^2
\\
    &&     + ||L(\xi_m-\xi_n)||^2.
\end{eqnarray*}
Therefore, as  $m,n\to\infty$, we have
\begin{eqnarray*}
   && ||\xi_m-\xi_n||^2  +  ||L(\xi_m-\xi_n)||^2
\\
   &=& 2 \left( G(\xi_m)+G(\xi_n)\right) - 4 G\left(  \frac{\xi_m+\xi_n}{2}   \right)
  \to 0.
\end{eqnarray*}
It follows that
\[
  \lim_{m,n\to\infty} ||\xi_m-\xi_n|| =0, \qquad
  \lim_{m,n\to\infty} ||L\xi_m-L\xi_n|| =0.
\]
Therefore, the sequence $\xi_n$ converges in norm to an element $\xi\in H_0$, and the sequence
 $L \xi_n$ converges in norm to an element $g\in H$. By the definition of the closed operator we have
 $\xi\in\DD(L)$ and $g=L\xi$. Moreover, we have
 \[
   \lim_{n\to\infty} ||L\xi_n||= ||L\xi||.
 \]
Finally, we obtain
\[
  \sigma^2 =\lim_{n\to\infty} G(\xi_n) = G(\xi), \qquad \xi\in H_0\cap \DD(L),
\]
as required.
\end{proof}


\begin{proof} [ of Proposition \ref{p:uniq}]
Assume that $\xi_1$ and $\xi_2$ are two distinct solutions
of the problem \eqref{A1}, i.e.
\[
   G(\xi_1)=G(\xi_2)= \sigma^2.
\]
By using the parallelogram identity \eqref{para}, we have
\begin{eqnarray*}
    && G\left( \frac{\xi_1+\xi_2}{2}\right)
\\  &=& \frac 12 \left(G(\xi_1)+G(\xi_2)\right)
  - \left\| \frac{\xi_1-\xi_2}{2} \right\|^2
  - \left\| \frac{L\xi_1-L\xi_2}{2} \right\|^2
\\
  &\le& \sigma^2 - \frac 14 \ \left\| \xi_1-\xi_2 \right\|^2
  <\sigma^2,
\end{eqnarray*}
which is impossible. The contradiction proves the proposition.
\end{proof}


\begin{proof}[ of Proposition \ref{p:LL}]
Let $L^*$ be the operator adjoint to $L$. Since $L$ is a closed
operator with the dense domain,  $L^*L$ is a self-adjoint
non-negative operator, cf.\,\cite[Chapter IV]{AG}. Therefore,
$(I+L^*L)^{-1}$ is a bounded self-adjoint operator. By letting
$A:=I+L^*L$, we have
\begin{eqnarray*}
  && ||y-x||^2+||Ly||^2 = ||x||^2-(x,y)-(y,x) +(Ay,y)
\\
  &=&     ||x||^2 + ||A^{1/2} y-A^{-1/2} x||^2 - ||A^{-1/2} x||^2
\\
   &\ge& ||x||^2 - ||A^{-1/2} x||^2,
\end{eqnarray*}
and the equality is attained iff $y=\xi=A^{-1}x$.
\end{proof}


\begin{proof}[ of Proposition \ref{p:euler}]
Let $\xi$ be a solution of Problem \eqref{A1}. Then for every
$Y\in H_0\cap \DD(L)$ we have
\[
   m:=||\xi-x||^2+||L\xi||^2\le ||Y-x||^2+||L Y||^2.
\]
Fix an arbitrary $h\in H_0\cap \DD(L)$; then $\xi+\eps h\in H_0\cap \DD(L)$
for all real $\eps$. We have
\begin{eqnarray*}
  m &\le& ||\xi+\eps h-x||^2+||L(\xi+\eps h)||^2
\\
   &=&  m + 2 \eps \left[ \Re(\xi-x,h) + \Re(L \xi, L h) \right]
\\
   &&  + \eps^2 \left[ ||h||^2+||L h||^2 \right].
\end{eqnarray*}
It follows that
\[
    \Re(\xi-x,h) + \Re(L \xi, L h) =0.
\]
By replacing $\eps$ with $i\eps$ we obtain
\[
    \Im(\xi-x,h) + \Im(L \xi, L h) =0.
\]
By adding up two equalities we arrive at \eqref{euler}.

We establish now the uniqueness of the solution for \eqref{euler}.
Let $\xi_1,\xi_2\in H_0\cap \DD(L)$ satisfy equalities
\begin{eqnarray*}
    && (\xi_1-x,h)+( L\xi_1,L h)=0,
\\
    && (\xi_2-x,h)+( L\xi_2,L h)=0,
\end{eqnarray*}
for all $h\in  H_0\cap \DD(L)$. It follows that
\[
  (\xi_1-\xi_2,h)+( L(\xi_1-\xi_2),L h)=0.
\]
By plugging $h=\xi_1-\xi_2$ into this equality, we arrive at
\[
   ||\xi_1-\xi_2||^2+ ||L(\xi_1-\xi_2)||^2=0,
\]
hence, $\xi_1=\xi_2$.
\end{proof}

\section*{Acknowledgements} M.A. Lifshits work was supported by RFBR grant
16-01-00258.


\begin{thebibliography}{99}

{\baselineskip=10pt

\bibitem{Akh} Akhiezer, N.I. Lectures on Approximation Theory.
Nauka, 1965.

\bibitem{AG} Akhiezer, N.I., Glazman, I.M.
Theory of Linear Operators in Hilbert Space.  Dover, 1993.

\bibitem{Ber} Bernstein, S.N. Sur une propri\'et\'e de fonctions enti\`eres.
Comptes Rendus Acad. Sci. Paris, t.176, 1603--1605 (1923).


\bibitem{Doob1} Doob, J.L.
Stochastic Processes. Wiley, N.Y. (1990).

\bibitem{Doob2} Doob, J.L.
Time series and harmonic analysis. In: Proc. Berkeley Sympos. Math. Statist. Probab.,
303--344, Berkeley--Los Angeles, University of California Press, 1949.

\bibitem{GS} Grenander, U., Szeg\"o, G.
Toeplitz Forms and Their Applications. AMS Chelsea Publ. (1998).

\bibitem{Hof} Hoffman, K.
Banach Spaces of Analytic Functions, Prentice-Hall, Englewood Cliffs, N.J. (1962).


\bibitem{KajTaq}
Kaj, I., Taqqu, M. S. Convergence to fractional Brownian motion and to the Telecom process:
the integral representation approach. In: In and Out of Equilibrium 2.
Ed. V. Sidoravicius and M. E. Vares. Ser: Progress in Probability {\bf 60},
Basel: Birkh\"auser Verlag, 2008, pp. 383--427.

\bibitem{KabLi}
Kabluchko, Z., Lifshits M.A. Least energy approximation for processes
with stationary increments. Preprint {\tt arxiv 1506.08369}. To appear in
J. Theor. Probab.

\bibitem{Kolm} Kolmogorov A.N.
Interpolation and extrapolation of stationary random sequences. In:
Selected Works of A.N. Kolmogorov, v.II, Probability and mathematical Statistics.
Kluwer, Dordreht, 1992. Original Russian edition: Izvestia AN SSSR. Ser. Math.,
v.5, pp. 3-14, 1941.


\bibitem{RPE} Lifshits, M. Random Processes by Example, World Scientific, Singapore,
2015.


\bibitem{PW} Paley, R.E.A.C., Wiener, N.
Fourier Transformations in the Complex Domain, AMS, Providence,
1934.

\bibitem{Roz} Rozanov Yu.A.
Stationary Random Processes, Holden--Day, 1967. Original Russian edition:
Nauka, Moscow, 1963.

\bibitem{Ru} Rudin, W. Real and Complex Analysis. Third Edition. Mc.Graw--Hill, 1987.

\bibitem{Yagl} Yaglom, A.M.  An Introduction to the Theory of Stationary
Random Functions. Revised English edition.
Dover, New York, 2004.

\bibitem{W} Wiener, N.
Extrapolation, Interpolation, and Smoothing of Stationary Time Series,
Technology Press of MIT, Cambridge, 1949.
}

\end{thebibliography}
\end{document}